\documentclass[12pt]{amsart} 
 \normalfont

\usepackage[psamsfonts]{amssymb}
\usepackage{pifont}
\theoremstyle{plain}

\input{xy}
\xyoption{all}

\newtheorem{theorem}{Theorem}[section]
\newtheorem{lemma}[theorem]{Lemma}
\newtheorem{proposition}[theorem]{Proposition}
\newtheorem{corollary}[theorem]{Corollary}

\def\bc{\mathbb{C}}

\def\br{\mathbb{R}}
\def\bh{\mathbb{H}}

\def\fA{\mathfrak{A}}

\def\fB{\mathfrak{B}}
\def\fC{\mathfrak{C}}

\def\fG{\mathfrak{G}}
\def\fH{\mathfrak{H}}
\def\fI{\mathfrak{I}}

\def\fZ{\mathfrak{Z}}

\newcounter{commentlabel}
\begin{document}
  \title{Isometries of Clifford Algebras I}
\author{Patrick Eberlein}
\address{Department of Mathematics, University of North Carolina, Chapel Hill, NC 27599}
\email{pbe@email.unc.edu}
\date{\today}
\maketitle

\noindent $\mathbf{Abstract}$  Let V be a finite dimensional vector space over a field F of characteristic $\neq 2$, and let Q be a nondegenerate, symmetric, bilinear form on V.  Let $C\ell(V,Q)$ be the Clifford algebra determined by V and Q.  The bilinear form Q extends in a natural way to a nondegenerate, symmetric bilinear form $\overline{Q}$ on $C\ell(V,Q)$.  Let G be the group of isometries of $C\ell(V,Q)$ relative to $\overline{Q}$, and let $\fG$ be the Lie algebra of infinitesimal isometries of $C\ell(V,Q)$ relative to $\overline{Q}$.  We derive some basic structural information about $\fG$, and we compute G in the case that $F = \br, V = \br^{n}$ and Q is positive definite on $\br^{n}$.  In a sequel to this paper we determine $\fG$ in the case that $F = \br, V = \br^{n}$ and Q is nondegenerate on $\br^{n}$.
\newline

\noindent $\mathbf{Remark}$  The main result of this preprint is known, but is left to the reader as a "hard exercise".  We would like to leave this preprint on the ArXiv in case the details of one solution are of interest to the reader.  
\newline

\noindent I was informed by the referee of "Isometries of Clifford Algebras II" that the main results of "Isometries of Clifford Algebras I" and Isometries of "Clifford Algebras, II" are both contained in the table labeled $R_{p,q}$ at the top of page 271 of I. Porteus, Topological Geometry, Cambridge University Press, Cambridge, 1969.  This result of Porteous also appears later in table XB on page 737 of the paper "Scalar products of spinors and an extension of Brauer-Wall groups" by P. Lounesto.  This paper was published in Foundations of Physics, vol. 11, Nos. 9/10, 721-740. 
\newline

\noindent $\mathbf{Introduction}$   Let V be a finite dimensional vector space over a field F with characteristic $\neq 2$.  Let $Q : V \times V\rightarrow F$ be a nondegenerate, symmetric, bilinear form on V.  Let $C\ell(V,Q)$ denote the Clifford algebra determined by V and Q, where the multiplication in V is given by $v^{2} = - Q(v,v)$.  There is a canonical involutive anti-automorphism c of $C\ell(V,Q)$ such that $c \equiv - Id$ on V.  The nondegenerate, symmetric, bilinear  form Q on V extends in a canonical way to a nondegenerate, symmetric, bilinear form $\overline{Q}$ on $C\ell(V,Q)$.
\newline 

\noindent  In the case that $F = \br$ the Clifford algebras are algebra isomorphic to matrix algebras $M(p,K)$ or $M(p,K) \oplus M(p,K)$, where $K = \br, \bc$ or $\bh$ and p depends on dim V.  The Clifford algebras have played an important role in mathematics and physics.  See [LM] for further details.
\newline

\noindent Let $G = \{g \in C\ell(V,Q) : L_{g}~ \rm{and}~R_{g}~ \rm{preserve~\overline{Q}} \}$, and let $\fG = \{\xi \in C\ell(V,Q) : c(\xi) = - \xi \}$.  These definitions make sense for any field F of characteristic $\neq 2$, and $\fG$ is the Lie algebra of G in the case that $F = \br$ or $\bc$.  In this paper we determine some of the basic structure of $\fG$.  If $dim~V \neq 1~(mod~4)$, then the Killing form B is nondegenerate on $\fG$.  If $dim~V \equiv 1~(mod~4)$, then the center $\fZ(\fG)$ is 1-dimensional and $\fG = \fZ(\fG) \oplus \fH$, where $\fH$ is an ideal of $\fG$ and B is nondegenerate on $\fH$.   The bilinear forms B and $\overline{Q}$ have a common orthogonal basis of $\fG$.
\newline

\noindent In the special case that $V = \br^{n}$ and the symmetric, bilinear form Q is positive definite on $\br^{n}$ we compute the group G.  The group G is compact and isomorphic to $H$ or $H \times H$, where $H = \{g \in M(p,K) : g \cdot g^{*} = g^{*} g = 1 \}$ for some positive integer p that depends on n and some choice of $K = \br, \bc$ or $\bh$.
\newline

\noindent   Let nonnegative integers r,s be given.  Let $C\ell(r,s)$ be the Clifford algebra when $V = F ^{r+s}$ and Q has the matrix $\left[\begin{array}{ccc}I_{r} & 0 \\ 0 & - I_{s}\\ \end{array} \right]$ for a suitable basis $\{v_{1}, ... , v_{r+s} \}$ of $F^{r+s}$.  $C\ell(r,s)$ is uniquely determined  up to algebra isomorphism by r and s.  In this paper we consider the case that $F = \br$ and $s = 0$.  In a subsequent paper we compute $\fG_{r,s} = \{\xi \in C\ell(r,s) : c(\xi) = - \xi \}$ when $s > 0$ and $F = \br$.
\newline

\noindent  Here is a brief description by section of the paper.  In section 1 we present some basic well known facts about Clifford algebras and include some proofs for the sake of completeness.  In section 2 we discuss the equivalence of bilinear forms $Q_{1},Q_{2}$ on V and the algebra isomorphism between $C\ell(V,Q_{1})$ and $C\ell(V,Q_{2})$ when $Q_{1}$ and $Q_{2}$ are equivalent.  In section 3 we discuss the relation between zero divisors and invertible elements in $C\ell(V,Q)$.  In section 4 we define and discuss the canonical nondegenerate, bilinear form $\overline{Q}$ on $C\ell(V,Q)$ that extends Q on V.  The bilinear form $\overline{Q}$ is characterized by a small set of axioms.  In section 5 we determine the centers of $C\ell(V,Q$ and $\fG$.  In section 6 we discuss the group of isometries G on $C\ell(V,Q)$ and the Lie algebra $\fG$ of infinitesimal isometries $\fG$ of $C\ell(V,Q)$.  We determine the Killing form B of $\fG$.  In section 7 we discuss some quaternionic linear algebra needed later.  In section 8 we list the matrix algebras that are algebra isomorphic to $C\ell(\br^{n},Q)$, where Q is a symmetric, positive definite bilinear form on $\br^{n}$.  In section 9 we find a special isomorphism between $C\ell(\br^{n},Q)$ and a matrix algebra A, and we use this isomorphism to determine $G = \{g \in C\ell(\br^{n},Q) : L_{g}~ \rm{and}~R_{g}~ \rm{preserve~\overline{Q}} \}$. 

\section{Preliminaries}

\noindent  Good references for this section are [H], [LM, chapter 1] and [FH,pp.299-315].  
\newline

\noindent $\mathbf{Universal~ mapping~ definition~ of~ the~ Clifford~ algebras}$

\noindent  Let V be a finite dimensional vector space over a field F, which we shall always assume has characteristic $\neq 2$.  Let Q be a nondegenerate, symmetric bilinear form on V.  The next result is fundamental.  For a proof  see, for example, Proposition 1.1 of Chapter 1, section 1 of [LM].

\begin{proposition}  Let (V,Q) be as above.   Then there exists an F-algebra $C\ell(V,Q)$ and an injective linear map $i : V \rightarrow C\ell(V,Q)$ with the following property :   Let $\fA$ be any associative finite dimensional  algebra over F, and let $\sigma : V \rightarrow \fA$ be any F - linear map such that $\sigma(v) \cdot \sigma(v) = - Q(v,v)1$ for all v $\in$ V.  Then there exists a unique algebra homomorphism $j : C\ell(V,Q) \rightarrow \fA$ such that $j \circ i = \sigma$.
\end{proposition}

\noindent $\mathbf{Remark}$  The extension property of the result above characterizes $C\ell(V,Q)$ up to algebra isomorphism.
\newline

\noindent  In the sequel we may assume without loss of generality, since $i : V \rightarrow C\ell(V,Q)$ is injective, that  $i : V \rightarrow C\ell(V,Q)$ is the inclusion map.

\begin{corollary} Let  V be a finite dimensional vector space over a field F with characteristic $\neq 2$ , and let $Q : V \times V \rightarrow F$ be a nondegenerate, symmetric, bilinear form on V.  Let  $C \ell(V,Q)$ denote the Clifford algebra determined by V.  Then

	$1)~ xy + yx = - 2Q( x,y) $ \hspace{1in} for all $x,y \in V$.
	
	$2)$  Let $\fA$ be a subalgebra of $C\ell(V,Q)$ that contains V.  Then $\fA = C\ell(V,Q)$.
\end{corollary}

\begin{proof}   We prove (1).  First note that $x \cdot x = - Q(x,x)$ for all $x \in V$ by the construction of $C\ell(V,Q) = T(V) / I$ above.  Let x,y be elements of V.  We compute $x^{2} + y^{2} + xy + yx = (x+y)(x+y) = - Q(x+y , x+y) = - Q(x,x) - Q(y,y) - 2 Q(x,y) = x^{2} + y^{2} - 2 Q(x,y).$
\newline

\noindent  We prove (2).  Let $\sigma : V \rightarrow \fA$ be the inclusion map.  By Proposition 1.1 there exists a unique algebra homomorphism $j : C\ell(V,Q) \rightarrow \fA$ such that $j \circ i = \sigma$.  Let $k : \fA \rightarrow C\ell(V,Q)$ be the inclusion map.  Then k is an algebra homomorphism, and hence $\varphi = k \circ j : C\ell(V,Q) \rightarrow C\ell(V,Q)$ is an algebra homomorphism such that $\varphi \circ i  = i$.  It follows that $\varphi = Id$ on $C\ell(V,Q)$ by the uniqueness part of Proposition 1.1 applied to the linear map $i : V \rightarrow C\ell(V,Q)$.  In particular j is injective, which implies that $dim~C\ell(V,Q) \leq dim~\fA \leq dim~C\ell(V,Q)$.  We conclude that $\fA = C\ell(V,Q)$. 
\end{proof} 
	
\noindent $\mathbf{Example~1}$ Let $V = \br^{n}$ and let r,s be nonnegative integers such that $r+s = n$.  For elements $v = (v_{1}, v_{2}, ... , v_{n})$ and $w = (w_{1}, w_{2}, ... , w_{n})$ of $\br^{n}$ define $Q(v,w) = \sum_{i=1}^{r} v_{i} w_{i} - \sum_{j=1}^{s} v_{r+j} w_{r+j}$. The classical case first considered is when $r = n$ and $s = 0$.
\newline

\noindent  $\mathbf{Example~2}$ Let $V = \br^{2}$ and define $Q(v,w) = v \cdot w$, the dot product.  We assert that $C\ell(\br^{2},Q) = \bh$, the quaternions.  Let $\{e_{1}, e_{2} \}$ be the standard basis of $\br^{2}$.  If $i = e_{1}, j = e_{2}$ and $k = e_{1} \cdot e_{2}$, then we leave it as an exercise to show that $C\ell(\br^{2},Q) = \bh$.
 \newline
	
\noindent The Clifford algebra $C \ell(V,Q)$ becomes a Lie algebra with the bracket operation given by

\hspace{1in}	$ [x,y] =  xy - yx$ \hspace{1in} for all $x,y \in C\ell(V,Q)$
\newline
 
\noindent $\mathbf{Canonical~ automorphism~ \alpha~ of~ C\ell(V,Q)}$
\newline

\noindent Let $\alpha : V \rightarrow C \ell(V,Q)$ be the $\br$- linear map given by $\alpha(v) = - v$.   Then clearly $\alpha(v) \cdot \alpha(v) = v \cdot v = - Q(v,v) 1$ for all v $\in V$.  Hence $\alpha$ extends to an algebra homomorphism of $C\ell(V,Q)$ by Proposition 1.1.  If $\beta = \alpha \circ \alpha$, then $\beta = Id$ on V and $\beta(v) \cdot \beta(v) = v \cdot v = -~ Q(v,v) 1$ for all v $\in$ V.  By the universal mapping definition $\beta$ extends uniquely to an algebra homomorphism of $C\ell(V,Q)$.  Since the identity is such an extension of $\beta$ it follows by uniqueness that $\beta = \alpha \circ \alpha$ is the identity on $C\ell(V,Q)$.  Hence $\alpha$ is an automorphism of $C\ell(V,Q)$.
\newline

\noindent  $\mathbf{Canonical~ anti-automorphisms~ \tau, c~of~ C\ell(V,Q)}$
\newline

\noindent We define two anti-automorphisms $\tau$ and $c$ (transpose and conjugation) of $C\ell(V,Q)$  Before doing so we recall the definition of $\mathit{opposite}$ algebra.  Let A be an associative, finite dimensional algebra over $\br$.  The opposite algebra $\tilde{A}$ equals A as a vector space.  However, if $\mu : A \times A \rightarrow A$ is the multiplication operation in A, then we define a multiplication operation $\tilde{\mu} : \tilde{A} \times \tilde{A} \rightarrow \tilde{A}$ by $\tilde{\mu}(x,y) = \mu(y,x)$.  It is easy to check that $\tilde{\mu}$ is associative and that $\tilde{A}$ is an algebra.
\newline

\noindent We define the anti-automorphism $\tau$.  Let $\fC = C\ell(V,Q)$ and let $\tilde{\fC}$ denote the opposite algebra.  Let $\sigma : V \rightarrow \tilde{\fC}$ denote the inclusion map.  For every v $\in$ V $\sigma(v) \cdot \sigma(v) = v \cdot v = - Q(v,v)1$, and hence there exists a unique algebra homomorphism $\tau : \fC \rightarrow \tilde{\fC}$ such that $\tau \circ i = \sigma$.  It is easy to see that $\tau : \fC \rightarrow \fC$ is an anti-homomorphism that is the identity on V.  Moreover, $\tau^{2} = Id$ on $\fC$ since both the identity map and $\beta = \tau^{2}$ are automorphisms of $\fC$ that are the identity on V.  Hence $\tau$ is an anti-automorphism of $\fC$.
\newline

\noindent  We define $c = \alpha~ \circ~ \tau$.  Clearly c is an anti-automorphism of $\fC$.  Moreover, since c is $- Id$ on V, it follows that $c^{2} = Id$ since both the identity map and $\gamma = c^{2}$ are automorphisms of $\fC$ that are the identity on V.
\newline

\noindent $\mathbf{Action~ of~ \alpha, \tau~\rm}$ and $\mathbf{c~on~monomials}$
\newline

\noindent   Let $v_{1}, ... , v_{N}$ be arbitrary elements of V, and let $v = v_{1} \cdot ... \cdot v_{N}$.  Since $\alpha$ is an automorphism of $\fC$ while $\tau$ and c are anti-automorphsms we obtain immediately from the definitions
\newline

\noindent \hspace{.5in}  $\alpha(v) = (-1)^{N} v$

\noindent \hspace{.5in}  $\tau(v) = v_{N} \cdot v_{N-1} \cdot ... \cdot v_{2} \cdot v_{1} $

\noindent \hspace{.5in}  $c(v) = (-1)^{N} v_{N} \cdot v_{N-1} \cdot ... \cdot v_{2} \cdot v_{1} $
\newline

\noindent  As an immediate consequence of the statements above we see that any two of the maps $\{\alpha, \tau,c \}$ commute.
\newline

\noindent $\mathbf{Bases~ and~ dimension~ of~ C\ell(V,Q)}$
\newline

\noindent  Let V be a finite dimensional vector space over a field F of characteristic $\neq 2$.  A basis $\{e_{1}, e_{2}, ... , e_{n} \}$ of V is said to be $\mathit{Q-orthogonal}$ if $Q(e_{i}, e_{j}) = 0$ whenever $i \neq j$.  Note that $Q(e_{i}, e_{i}) \neq 0$ for all i since Q is nondegenerate.  A Q-orthogonal basis for V always exists by the nondegeneracy of Q (see for example Theorem 3.1 , Chapter XV, section 2 of [L]).
\newline

\noindent Let $p : C\ell(V,Q) \rightarrow F$ be the linear map such that p(x) is the component of x in F for all $x \in C\ell(V,Q)$.  Define a bilinear form $\overline{Q} : C\ell(V,Q)  \times C\ell(V,Q) \rightarrow F$ by $\overline{Q}(x,y) = p(x \cdot c(y))$ for all $x,y \in C\ell(V,Q)$.
\newline

\noindent Let $\fB = \{e_{1}, ... , e_{n} \}$ be an orthogonal basis of V relative to Q.  Let $\fI_{k}$ denote the set of multi-indices $I = (i_{1}, ... , i_{k})$, where $1 \leq i_{1} < i_{2} < ... < i_{k} \leq n$.  For each multi-index $I = (i_{1}, ... , i_{k}) \in \fI_{k}$ let $e_{I}$ denote the product $e_{i_{1}} \cdot e_{i_{2}} \cdot ... \cdot e_{i_{k}}$.  Let $\fI = \bigcup_{k=1}^{n} \fI_{k}$.   If $I \in \fI_{k}$, then set $|I| = k$.
\newline

\begin{proposition}  Let $\fB =  \{e_{1}, e_{2}, ... , e_{n} \}$ be any Q-orthogonal  basis  of V.  Let $\fB' = \{1 \} \cup \{e_{I} : I \in \fI \}$.  Then 

	1)  $\fB'$ is orthogonal relative to $\overline{Q}$.  Moreover, $\overline{Q}(x,x) \neq 0$ for all $x \in \fB^{\prime}$.

	2)  $\fB'$ is a basis of $C\ell(V,Q)$. 
	
	3)  $\overline{Q}$ is symmetric and nondegenerate on $C\ell(V,Q)$. 
	
\end{proposition}

\noindent $\mathbf{Remark}$  It follows from 2) that $dim~ C\ell(V,Q) = 2^{n}$.  We shall see below in Proposition 4.1 that $\overline{Q}$  is uniquely characterized by certain properties.

\begin{proof}  \noindent \noindent  We prove 1).  Let $x = e_{i_{1}} \cdot e_{i_{2}} \cdot ... \cdot e_{i_{r}}$ and $y = e_{j_{1}} \cdot e_{j_{2}} \cdot ... \cdot e_{j_{s}}$ be any elements of $\fB'$.  Then $\overline{Q}(x,y) = p(e_{i_{1}} \cdot e_{i_{2}} \cdot ... \cdot e_{i_{r}} \cdot c(e_{j_{s}}) \cdot ... \cdot c(e_{j_{1}})) = \pm p(e_{i_{1}} \cdot e_{i_{2}} \cdot ... \cdot e_{i_{r}} \cdot e_{j_{s}} \cdot ... \cdot e_{j_{1}}) = 0$ unless $r = s$ and $\{i_{1}, ... , i_{r} \} = \{j_{1} , ... , j_{r} \}$.  If $x = y$, then  $\overline{Q}(x,x) \neq 0$ since $Q(e_{i},e_{i}) \neq 0$ for every i.
\newline

\noindent We prove 2).   By 1) of Corollary 1.2 it follows that any finite product of elements from $\fB^{\prime}$ can be written as $c~ e_{i_{1}} \cdot e_{i_{2}} \cdot ... \cdot e_{i_{k}}$ for some integers $1 \leq i_{1} < i_{2} < ... < i_{k} \leq n$ and some nonzero constant $c \in F$.  It follows that if $\fA = \br-$ span $\fB^{\prime}$, then $\fA$ is a subalgebra of $C\ell(V,Q)$ that contains V.  From 2) of Corollary 1.2 we see that $\fA = C\ell(V,Q)$ and hence $\fB^{\prime}$ spans $C\ell(V,Q)$.  From 1) of this Proposition it follows that $\fB^{\prime}$ is linearly independent on $C\ell(V,Q)$.
\newline

\noindent  The assertion 3) follows immediately from assertions 1) and 2).

\end{proof}

\section{Equivalence of bilinear forms}

\noindent Let V be a finite dimensional vector space over a field F of characteristic $\neq 2$, and let $Q_{i} : V \times V \rightarrow F$ be symmetric bilinear forms for $i = 1,2$.  The bilinear forms $Q_{1}, Q_{2}$ are said to be $\mathit{equivalent}$ if there exists a nonsingular linear transformation $T : V \rightarrow V$ such that $Q_{2}(v,w) = Q_{1}(T(v), T(w))$ for all $v,w \in V$.  It is easy to see that being equivalent is an equivalence relation on the vector space of symmetric bilinear forms on V.

\begin{proposition}  Let V be a finite dimensional vector space over F, and let $Q_{1}, Q_{2}$ be equivalent, nondegenerate, symmetric, bilinear forms on V.  Then $C\ell(V,Q_{1})$ is algebra isomorphic to $C\ell(V,Q_{2})$.
\end{proposition}

\begin{proof}  By the equivalence of $Q_{1}$ and $Q_{2}$ there exists  a nonsingular linear transformation $T : V \rightarrow V$ such that $Q_{2}(v,w) = Q_{1}(T(v), T(w))$ for all $v,w \in V$.  Define $i : V \rightarrow C\ell(V, Q_{1})$ by $i(v) = T(v)$.  Clearly i is linear and injective.  Moreover, if $v \in V$, then $- Q_{2}(v,v) = - Q_{1}(T(v),T(v))= T(v)^{2} = i(v)^{2}$.  Hence by the universal mapping definition of $C\ell(V,Q_{1})$ the injective linear map i extends to an algebra homomorphism $\tilde{i} : C\ell(V, Q_{2}) \rightarrow C\ell(V, Q_{1})$.  The map $\tilde{i}$ is surjective since $\tilde{i}(V) = i(V) = V$ generates $C\ell(V,Q_{1})$ as an algebra by 2) of Corollary 1.2.  We conclude that the map $\tilde{i}$ is an isomorphism since the dimensions of $C\ell(V, Q_{1})$ and $C\ell(V, Q_{2})$ are the same.
\end{proof}

\begin{corollary}  Let $Q_{1},Q_{2} : \br^{n} \times \br^{n} \rightarrow \br$ be symmetric, positive definite, bilinear forms on $\br^{n}$.  Then $C\ell(\br^{n}, Q_{1})$ is algebra isomorphic to $C\ell(\br^{n}, Q_{2})$.
\end{corollary}

\begin{proof}  For every symmetric, positive  definite, bilinear form Q on $\br^{n}$ there exists a basis $\{v_{1}, v_{2}, ... , v_{n} \}$ of $\br^{n}$ such that $Q(v_{i} , v_{j}) = \delta_{ij}$ for $1 \leq i,j \leq n$.   Given symmetric, positive definite, bilinear forms $Q_{1}, Q_{2}$ on $\br^{n}$ choose orthonormal bases $\fB_{1}, \fB_{2}$ and a linear isomorphism T of $\br^{n}$ such that $T(\fB_{1}) = \fB_{2}$.  It follows that $Q_{1}(v,w) = Q_{2}(Tv,Tw)$ for all $v,w \in \br^{n}$.  Now apply Proposition 2.1.
\end{proof}

\section{Zero divisors and invertible elements}

\noindent In general zero divisors exist in $C\ell(V,Q)$, but only for a set of elements that lie in an F-algebraic variety in $C\ell(V,Q)$.
\newline

\noindent $\mathbf{Example}$  Let F be a field with characteristic $\neq 2$, and let $V = F^{3}$.  Let Q be the nondegenerate bilinear form on $F^{3}$ given by $Q(v,w) = v_{1} w_{1} + v_{2} w_{2} - v_{3} w_{3}$, where $v = (v_{1},v_{2},v_{3}) \in F^{3}$ and  $w = (w_{1},w_{2},w_{3}) \in F^{3}$.  If $x = e_{1} + e_{3}$, then $x^{2} = 0$.
\newline
 
 \noindent  We now give a more general description of the set $\fZ$ of zero divisors of $C\ell(V,Q)$. We then use this information to show that the set of  invertible elements  in $C\ell(V,Q)$ is $C\ell(V,Q) - \fZ$.
 \newline
 
 \noindent If a nonzero element a is a zero divisor, then either $0 = ax = L_{a}(x)$ for some nonzero element x or $0 = xa = R_{a}(x)$ for some nonzero element x of $C\ell(V,Q)$.  Define algebra homomorphisms $\varphi_{L} : C\ell(V,Q) \rightarrow F$ and $\varphi_{R} : C\ell(V,Q) \rightarrow F$  by $\varphi_{L}(a) = det~L_{a}$ and $\varphi_{R}(a) = det~R_{a}$.  We may write $\varphi_{L} = det \circ L$, where $L : C\ell(V,Q) \rightarrow End(C\ell(V,Q))$ is the algebra homomorphism give by $L(a) = L_{a}$.  The map $\varphi_{L}$ is a polynomial map since L is a linear map and det is a polynomial map.  The same argument shows that $\varphi_{R}$ is also a polynomial map .
\newline

\noindent  The discussion above shows that the set of zero divisors in $C\ell(V,Q)$ is the variety $\fZ = \varphi_{L}^{-1}(0) \cup  \varphi_{R}^{-1}(0)$.  Note that the maps $\varphi_{L}$ and $\varphi_{R}$ are not identically zero.  For example, if v is an element of V such that $Q(v,v) \neq 0$, then $L_{v^{2}} = R_{v^{2}} = - Q(v,v)~Id \neq 0$.  In particular $0 \neq \varphi_{L}(v^{2}) = \varphi_{R}(v^{2}) = (\varphi_{L}(v))^{2} = (\varphi_{R}(v))^{2}$.
\newline

\noindent An element a of $C\ell(V,Q)$ is $\mathit{invertible}$ if there exists an element z of $C\ell(V,Q)$ such that $a \cdot z = z \cdot a = 1$.  The set GL(V,Q)  of invertible elements  In $C\ell(V,Q)$ forms a group.  

\begin{proposition}  $GL(V,Q = C\ell(V,Q) - \fZ$, where $\fZ$ denotes the set of zero divisors in $C\ell(V,Q)$ .
\end{proposition}

\begin{proof}  If $\varphi_{L}(a) \neq 0$ for some element a of $C\ell(V,Q)$, then $L_{a}$ is surjective and there exists x $\in C\ell(V,Q)$ such that $1 = L_{a}(x) = a \cdot x$.  Similarly, if $\varphi_{R}(a) \neq 0$ for some element a of $C\ell(V,Q)$, then $R_{a}$ is surjective and there exists y $\in C\ell(V,Q)$ such that $1 = R_{a}(y) = y \cdot a$.   If $\varphi_{L}(a)$ and $\varphi_{R}(a)$  are both nonzero, then $x = y$ and a is invertible.  Hence $C\ell(V,Q) - \fZ \subset GL(V,Q)$.  Conversely, let $a \in C\ell(V,Q)$ be invertible, and let x $\in C\ell(V,Q)$ be an element such that $1 = a \cdot x = x \cdot a$.  If $\varphi_{L}(a) = 0$, then $ 0 = L_{a}(z) = a \cdot z$ for some nonzero element z of $C\ell(V,Q)$.  But then $0 = x \cdot (a \cdot z) = (x \cdot a) \cdot z = z$, a contradiction.  Hence $\varphi_{L}(a) \neq 0$, and a similar argument shows that $\varphi_{R}(a) \neq 0$.  This shows that $GL(V,Q) \subset C\ell(V,Q) - \fZ$.
\end{proof}

\noindent
 
\section{Canonical symmetric bilinear form}

\noindent  Let F be a field with characteristic $\neq 2$.  Let V be a finite dimensional vector space over F, and let $Q : V \times V \rightarrow F$ be a symmetric, nondegenerate, bilinear form.  If $T : V \rightarrow V$ is a linear transformation, then let the metric adjoint $T^{*} : V \rightarrow V$ be the unique linear transformation such that $Q(Tv,w) = Q(v,T^{*}w)$ for all v,w $\in V$.  The existence and uniqueness of $T^{*}$ follows from the nondegeneracy of Q.
\newline

\noindent Let $c = \alpha \circ \tau$ be the anti-automorphism of $C\ell(V,Q)$ defined above.  Let  $\overline{Q}$ be the nondegenerate, symmetric, bilinear form on $C\ell(V,Q)$ discussed in Proposition 1.3. For $x \in C\ell(V,Q)$ let  $L(x) : C\ell(V,Q) \rightarrow C\ell(V,Q)$ and $R(x) : C\ell(V,Q) \rightarrow C\ell(V,Q)$ denote left and right translation by x respectively.  Let $L(x)^{*}$ and $R(x)^{*}$ denote the metric adjoints of $L(x)$ and $R(x)$ respectively relative to $\overline{Q}$,
\newline

\begin{proposition}    $\overline{Q}$  satisfies the following properties : 

	1)  $\overline{Q}(1,1) = 1$.
	
	2)  $L(x)^{*} = L(c(x))$  \hspace{.3in} for all x $\in C\ell(V,Q)$ 
	
	3)  $R(x)^{*} = R(c(x))$  \hspace{.3in} for all x $\in C\ell(V,Q)$ 
		
	4)  $\overline{Q} = Q$ on V.
\newline
	
	Moreover, $\overline{Q}$ satisfies
	
	5)  $\overline{Q}(\alpha(x), \alpha(y)) = \overline{Q}(x,y)$ for all x,y $\in C\ell(V,Q)$.
	
	6)  $\overline{Q}(c(x) , c(y)) = \overline{Q}(x,y)$ for all x,y $\in C\ell(V,Q)$.

\end{proposition}

\begin{proof}  We need some preliminary results and notation.  
\newline

\noindent Recall from Proposition 1.3 that $\fB' = \{1, e_{I} : I \in \fI  \}$ is a basis of $C\ell(V,Q)$.  For an element x of $C\ell(V,Q)$ we write $x = a + \sum_{I} x_{I} e_{I}$, where the sum is over all multi-indices $I \in \fI$, and $a = p(x), x_{I}$ are elements of F.

\begin{lemma}  Let $I \in \fI$ and $J \in \fI$ be given.  Then

	1) $0 \neq e_{I}^{2} = (-1)^{\frac{k(k-1)}{2}}~e_{i_{1}}^{2} \cdot e_{i_{2}}^{2} \cdot ...\cdot e_{i_{k}}^{2} \in F$.
	
	2)  $e_{I} \cdot e_{J} = \pm e_{J} \cdot e_{I} = \lambda_{IJ} e_{K}$, where $K = (I \cup J) - (I \cap J)$ if $I \neq J$ and $0 \neq \lambda_{IJ} \in F$.
\end{lemma}

\begin{proof}  Assertion 1) follows from induction on $|I|$.  Assertion 2) follows from assertion 1) and induction on $|I| + |J|$.
\end{proof}

\begin{lemma}  1)  $p(xy) = p(yx)$ for all x,y $\in C\ell(V,Q)$. 

\hspace{.7in} 2) $p(\alpha(x)) = p(x) = p(c(x)) $ for all x $\in C\ell(V,Q)$.
\end{lemma}  

\begin{proof}  We prove 1). Let x,y be elements of $C\ell(V,Q)$.  Write  $x = a + \sum_{I \in \fI} a_{I} e_{I}$ and  $y = b + \sum_{J \in \fI} b_{J} e_{J}$, where $a,b,a_{I}, b_{J} \in F$. Then $p(xy) = p(ab + \sum_{I,J \in \fI} a_{I} b_{J} (e_{I} \cdot e_{J}) + b~\sum_{I \in \fI} a_{I} e_{J} + a~\sum_{J \in \fI} b_{J} e_{J})$.  By Lemma 4.2 we have  $e_{I} e_{J} = \lambda_{IJ} e_{K}$, where $K = (I \cup J) - (I \cap J)$ and $\lambda_{IJ}$ is a nonzero element of F.  By inspection K is nonempty unless $I = J$.  If $I = J$, then $e_{I} e_{J} = e_{I} e_{I} = \epsilon_{I}$, where $0 \neq \epsilon_{I} \in F$ by Lemma 4.2. Hence $p(xy) = ab + \sum_{I} a_{I} b_{I} \epsilon_{I}$, where the sum is over all multi-indices $I \in \fI $.  This expression is symmetric in the components of x and y, which proves 1). 
\newline 

\noindent We prove 2)  Let x $\in C\ell(V,Q)$ be given, and write $x = a + \sum_{I} a_{I} e_{I}$  as in the proof of 1).  Both statements in 2) are now an immediate consequence of the fact that  the mappings $\alpha$ and c fix the elements of F and take $e_{I}$ into  $\pm e_{I}$ for every multi-index I.  
\end{proof}

\noindent  We are now ready to prove Proposition 4.1.  Assertion 1) is obvious.  We prove 2).   For elements x,y,z of $C\ell(V,Q)$ we have $\overline{Q}(L(x)y,z) = p(x \cdot y \cdot c(z))$.  On the other hand $\overline{Q}(y, L(c(x))z) = \overline{Q}(y, c(x) \cdot z) = p(y \cdot c(z) \cdot x) = p(x \cdot y \cdot c(z)) = \overline{Q}(L(x)y,z)$.  Hence $L(x)^{*} = L(c(x))$.  A similar argument proves 3).
\newline

\noindent  We prove 4).  If $v \in V$, then $\overline{Q}(v,v) = p(v \cdot c(v)) = - p(v \cdot v) =  - p(- Q(v,v)1) = Q(v,v)$.  A standard polarization argument now shows that $\overline{Q}(v,w) = Q(v,w)$ for all $v,w \in V$.
\newline

\noindent  We prove 5).  For elements $x,y \in C\ell(V,Q)$ we have $\overline{Q}(\alpha(x), \alpha(y)) = p(\alpha(x) \cdot c(\alpha(y))) = p(\alpha(x) \cdot \alpha(c(y)) ) = p(\alpha(x \cdot c(y))) = p(x \cdot  c(y)) = \overline{Q}(x,y)$.
\newline

\noindent We prove 6).  For elements $x,y \in C\ell(V,Q)$ we have  $\overline{Q}(c(x), c(y)) = p(c(x) \cdot y) = p(c(c(x) \cdot y)) = p(c(y) \cdot x) = p(x \cdot c(y)) = \overline{Q}(x,y)$.
\end{proof}

\noindent  We now prove a converse to Proposition 4.1.

\begin{proposition}  Let $\tilde{Q} : C\ell(V,Q) \times C\ell(V,Q) \rightarrow F$ be a symmetric bilinear form that satisfies properties 1),2),3) and 4) of Proposition 4.1.  Then $\tilde{Q} = \overline{Q}$
\end{proposition}

  \noindent Let $\overline{Q}$ be the nondegenerate, symmetric bilinear form defined on $C\ell(V,Q)$ above.  Let $\tilde{Q}$ be another symmetric, bilinear form on $C\ell(V,Q)$ that satisfies properties  1), 2), 3) and 4).  Since $\overline{Q}$ is nondegenerate on $C\ell(V,Q)$ there exists a linear transformation $S : C\ell(V,Q) \rightarrow C\ell(V,Q)$ such that $\tilde{Q}(x,y) = \overline{Q}(S(x),y)$ for all x,y $\in C\ell(V,Q)$.  It is easy to see that S is symmetric with respect to both $\overline{Q}$ and $\tilde{Q}$.
\newline

\begin{lemma}  $x \cdot S(y) = S(x \cdot y) = S(x) \cdot y$ for all x,y $\in C\ell(V,Q)$.
\end{lemma}

\begin{proof}  Let elements x,y,z $\in C\ell(V,Q)$ be given.  Since $\tilde{Q}$ satisfies property 3) it follows that $\overline{Q}(S(x \cdot y),z) = \tilde{Q}(x \cdot y,z) = \tilde{Q}(R(y)x,z) = \tilde{Q}(x, R(c(y))z) = \overline{Q}(S(x),R(c(y))z) = \overline{Q}(R(y)S(x),z) = \overline{Q}(S(x) \cdot y,z)$.  Since $\overline{Q}$ is nondegenerate we conclude that 

	1)  $S(x \cdot y) = S(x) \cdot y$ for all x,y $\in C\ell(V,Q)$
	
\noindent Similarly, since $\tilde{Q}$ satisfies property 2) we conclude that $\overline{Q}(S(x \cdot y),z) = \tilde{Q}(x \cdot y,z) = \tilde{Q}(L(x)y,z) = \tilde{Q}(y, L(c(x)) z) = \overline{Q}(S(y), L(c(x))z) = \overline{Q}(L(x)S(y),z) = \overline{Q}(x \cdot S(y),z)$.  From the nondegeneracy of $\overline{Q}$ we conclude

	2)  $S(x \cdot y) = x \cdot S(y)$ for all x,y $\in C\ell(V,Q)$
	
\noindent  This completes the proof of the lemma.
\end{proof}

\begin{lemma}  If $\lambda = S(1)$, then $\lambda$ lies in the center of $C\ell(V,Q)$, and $S(x) = \lambda x = x \lambda$ for all x $\in C\ell(V,Q)$.
\end{lemma}

\begin{proof}  We note that $S(y) = \lambda y$ for all y $\in C\ell(V,Q)$ by substituting $x=1$ in equation 1) above.  Similarly we observe that $S(x) = x \lambda$ for all x $\in C\ell(V,Q)$ by substituting $y=1$ in equation 2) above.
\end{proof}

\begin{lemma} $\lambda v = v$ for all $v \in V$.
\end{lemma}

\begin{proof}  Let v,w be elements of V.  Then by property  4) and Lemma 4.7 $\overline{Q}(v,w) = Q(v,w) = \tilde{Q}(v,w) = \overline{Q}(\lambda v,w)$, and hence $0 = \overline{Q}(\lambda v - v,w)$ for all $v,w \in V$.  The assertion of the lemma now follows from the nondegeneracy of $\overline{Q}$.
\end{proof}

\noindent  We now complete the proof of Proposition 4.4.  Let v be an element of V such that $Q(v,v) \neq 0$.  By the previous lemma it follows that $\lambda~Q(v,v) = - \lambda~v \cdot v = - v \cdot v = Q(v,v) 1$.  Hence $\lambda = 1$ and $\overline{Q} = \tilde{Q}$ since S is the identity by Lemma 4.6.

\section{Centers of the Clifford algebras}

\begin{proposition} Let Q be a nondegenerate, symmetric, bilinear form on a vector space V of dimension n over a field F of characteristic $\neq 2$. Let  $C\ell(V,Q)$ denote the corresponding Clifford algebra, and let  $\fZ(V,Q)$ denote the center of $C\ell(V,Q)$.  Let $\{ e_{1}, ... , e_{n}\}$ denote a Q-orthogonal basis of V and let $\omega = e_{1} \cdot ... \cdot e_{n}$.  Then

	1) If n is even, then
	
\hspace{.7in} 	a) $ \varphi \omega = \omega \alpha(\varphi)$ for all $\varphi \in C\ell(V,Q)$.
	
\hspace{.7in}	b)  $\fZ(V,Q) = F$.
	
	2)  If n is odd, then 
	
\hspace{.7in}	$\fZ(V,Q) = F-\rm{span} \{1,\omega\}$.
	
	3) $\omega^{2} = (-1)^\frac{n(n-1)}{2} e_{1}^{2} \cdot ... \cdot e_{n}^{2}$
	
\end{proposition}

\noindent  We begin with some preliminary results.

\begin{lemma}  Let $\xi \in \fZ(V,Q)$ and write $\xi = \xi_{0} + \sum_{I \in \frak{I}} \xi_{I} e_{I}$, where $\xi_{0} \in F$ and $\xi_{I} \in F$ for all I.  Then $e_{I} \in \fZ(V,Q)$ if $\xi_{I} \neq 0$.
\end{lemma} 

\begin{proof} Let $I \in \fI$ with $\xi_{I} \neq 0$ and k with $1 \leq k \leq n$ be given.  By Lemma 4.2 $\pm e_{I} e_{k} = e_{k} e_{I} = \lambda_{k} e_{I_{k}}$ for some $0 \neq \lambda_{k} \in F$, where $I_{k} = I \cup \{k \}$ if $k \notin I$ and $I_{k} = I - \{k \}$ if $k \in I$.  By inspection $I_{k} \neq J_{k}$ if $I \neq J \in \fI$. By Proposition 1.3 we have $\xi_{I} \overline{Q}(e_{k} e_{I} , e_{I_{k}}) = \overline{Q}(e_{k} \xi , e_{I_{k}}) = \overline{Q}(\xi e_{k} , e_{I_{k}}) = \xi_{I} \overline{Q}(e_{I} e_{k} , e_{I_{k}}) = \epsilon~\xi_{I} \overline{Q}(e_{k} e_{I} , e_{I_{k}})$, where $ \epsilon = \pm 1$ and $e_{k} e_{I} = \epsilon~ e_{I} e_{k}$.  Note that $\overline{Q}(e_{k} e_{I} , e_{I_{k}}) = \lambda_{k} \overline{Q}(e_{I_{k}} , e_{I_{k}}) \neq 0$ by 1) of Proposition 1.3.  Hence $\epsilon = 1$ since $\xi_{I} \neq 0$.
\newline

\noindent We have shown that $e_{k} e_{I} = e_{I} e_{k}$ for $1 \leq k \leq n$.  Hence $e_{I} \in \fZ(V,Q)$ since $e_{I}$ commutes with the elements in the set $\{e_{1}, ... , e_{n} \}$, which generates $C\ell(V,Q)$ as an algebra.
\end{proof}

\begin{lemma}  If $e_{I} \in \fZ(V,Q)$ for some $I \in \frak{I}$, then $e_{I} = \omega$ and n is odd. 
\end{lemma}

\begin{proof}  Let $I = (i_{1}, i_{2}, ... , i_{k})$, where $i_{j} < i_{j+1}$ for $1 \leq j \leq k-1$.  We show first that k is odd.  We compute $e_{i_{1}} e_{I} = (e_{i_{1}})^{2} e_{i_{2}} ... e_{i_{k}}$ and $e_{I} e_{i_{1}} = (-1)^{k-1} (e_{i_{1}})^{2} e_{i_{2}} ... e_{i_{k}}$.  Hence k is odd since $e_{i_{1}} e_{I} = e_{I} e_{i_{1}}$.
\newline

\noindent  Next suppose that $|I| = k < n$.  Choose $\alpha \in \{1,2, ... , n \}$ with $\alpha \notin \{i_{1}, i_{2}, ... , i_{k} \}$.   Then $e_{\alpha} e_{I} = e_{\alpha} e_{i_{1}} e_{i_{2}} ... e_{i_{k}}$ and $e_{I} e_{\alpha} = (-1)^{k} e_{\alpha} e_{i_{1}} e_{i_{2}} ... e_{i_{k}} = -1~e_{\alpha} e_{i_{1}} e_{i_{2}} ... e_{i_{k}} = - e_{\alpha} e_{I}$ since k is odd.  This contradiction shows that $|I| = k = n$ and $e_{I} = \omega \in \fZ(V,Q)$.  Finally, $n = k$ is odd, as was shown above.
\end{proof}

\noindent  We now prove the Proposition.  Assertion 1b) follows immediately from Lemmas 5.2 and 5.3. To prove 1a) we show first  that $\omega e_{i} = - e_{i} \omega$ for $1 \leq i \leq n$ if n is even.  We compute $\omega e_{i} = (-1)^{n-i} e_{1} e_{2} ... e_{i-1} e_{i}^{2} e_{i+1} ... e_{n}$ and $- e_{i} \omega = (-1)^{i} e_{1} e_{2} ... e_{i-1} e_{i}^{2} e_{i+1} ... e_{n} = \omega e_{i}$ since n  even implies that $n - 2i$ is even.    Next, we let $\fA = \{\varphi \in C\ell(V,Q) :   \varphi \omega = \omega \alpha(\varphi) \}$.  Note that $\fA$ is a subalgebra of $C\ell(V,Q)$ since $\alpha$ is an automorphism of $C\ell(V,Q)$.  It follows that $\fA = C\ell(V,Q)$ by 2) of Corollary 1.2 since $\fA \supset V$ by the work above.
\newline

\noindent Assertion 3) follows routinely by induction on n.  We prove  2).  Let n be odd.  Then  $\fZ(V,Q) \subset F-\rm{span}\{1,\omega\}$ by Lemmas 5.2 and 5.3.  To complete the proof of 2) it suffices to show that $\omega \in \fZ(V,Q)$ if n is odd.  This follows by the same argument used to prove 1a), namely by showing that $\omega e_{k} = e_{k} \omega$ for $1 \leq k \leq n$.

\section{Isometries and infinitesimal isometries}

\noindent $\mathbf{Isometries~ of~ C\ell(V,Q)}$

\begin{proposition}   Consider the following groups :

\noindent  $U(V,Q) = \{g \in C\ell(V,Q) : g \cdot c(g) = c(g) \cdot  g = 1 \}$.  

\noindent $U_{L}(V,Q) = \{g \in C\ell(V,Q) : L(g) ~\rm{preserves}~ \overline{Q} \}$ 

\noindent $U_{R}(V,Q) = \{g \in C\ell(V,Q) : R(g) ~\rm{preserves}~ \overline{Q} \}$.  

\noindent Then $U(V,Q) = U_{L}(V,Q) \cap  U_{R}(V,Q)$.
\end{proposition} 

\noindent $\mathbf{Remark}$  We regard the group $G = U(V,Q)$ as the group of isometries of $C\ell(V,Q)$. 

\begin{proof}  We show first that U(V,Q) is a subgroup of both $U_{L}(V,Q)$ and $U_{R}(V,Q)$.  Let g be an element of U(V,Q).  Let x,y be elements of $C\ell(V,Q)$.  Then $\overline{Q}(L(g) x ,L(g) y) = \overline{Q}(x, L(c(g)) L(g)(y)) = \overline{Q}(x, L (c(g) \cdot g) y) = \overline{Q}(x,y)$.  This shows that $U(V,Q) \subset U_{L}(V,Q)$ and a similar proof shows that $U(V,Q) \subset U_{R}(V,Q)$.
\newline

\noindent  We now show that $ U_{L}(V,Q) \cap  U_{R}(V,Q) \subset U(V,Q)$.  Let g $\in U_{L}(V,Q) \cap U_{R}(V,Q)$ and  x,y $\in C\ell(V,Q)$ be given.  Since $g \in U_{L}(V,Q)$ it follows that  $\overline{Q}(x,y) = \newline \overline{Q}(L(g) x, L(g) y) = \overline{Q}(x, L(c(g))L(g) y) = \overline{Q}(x, L(c(g) \cdot g) y) = \overline{Q}(x, c(g) \cdot g \cdot y)$.   Hence $0 = \overline{Q}(x, [1 - c(g) \cdot g] \cdot y)$ for all x,y $\in C\ell(V,Q)$.  By the nondegeneracy of $\overline{Q}$ it follows that  $[1 - c(g) \cdot g] \cdot  y = 0$ for all y $\in C\ell(V.Q)$.  Setting $y=1$ we conclude that $c(g) \cdot g =1$.  A similar argument shows that $g \cdot c(g) =1$ since $g \in U_{R}(V,Q)$.  Hence $g \in U(V,Q)$.
\end{proof}

\noindent $\mathbf{Infinitesimal~ isometries~ of~ C\ell(V,Q)}$

\noindent  We define $\fG = \{\xi \in C\ell(V,Q) : c(\xi) = - \xi \}$, where $c : C\ell(V,Q) \rightarrow C\ell(V,Q)$ is the canonical anti-automorphism.  The set $\fG$ is a Lie algebra since c is an anti-automorphism. For reasons we now explain, we call $\fG$ the Lie algebra of infinitesimal isometries of $C\ell(V,Q)$ relative to $\overline{Q}$.   
\newline

\noindent $\mathbf{Remark}$ Let G denote U(V,Q).  If $F =  \br$ or $\bc$, then the matrix exponential $exp : C\ell(V,Q) \rightarrow C\ell(V,Q)$ can be defined in the usual way.  If $\xi \in C\ell(V,Q)$, then the definition of G  and standard arguments show that $\xi \in \fG \Leftrightarrow exp(t \xi) \in G~\rm{for~all}~ t \in \br$.  Both G = U(V,Q) and the Lie algebra $\fG$ have been defined for any field F with characteristic $\neq 2$, but the relationship stated above between G and $\fG$ makes no sense for an arbitrary field F.  Nevertheless, it seems reasonable to think of $\fG$ as the Lie algebra of infinitesimal isometries of $C\ell(V,Q)$.
\newline

\noindent $\mathbf{Center~\fZ(\fG)~ of~ \fG}$
\newline

\noindent  Let $\fZ(\fG) = \{\xi \in \fG : ad~\xi = 0~on~\fG \}$, the center of $\fG$.

\begin{proposition}  $\fZ(\fG) = \fZ(V,Q) \cap \fG$.
\end{proposition}

\begin{proof}  \noindent  Clearly $\fZ(V,Q) \cap \fG \subset \fZ(\fG)$.  Now let $\xi \in \fZ(\fG)$ be a nonzero element, and let $\fH = \{\eta \in C\ell(V,Q) : \xi \cdot \eta = \eta \cdot \xi \}$.  It is easy to check that $\fH$ is a subalgebra of $C\ell(V,Q)$, and $\fH \supset V$ since $V \subset \fG$.  Hence $\fH = C\ell(V,Q)$ since V generates $C\ell(V,Q)$ as an algebra by 2) of Corollary 1.2.  We conclude that  $\xi \in \fZ(V,Q)$.
\end{proof}

\begin{corollary}  If n is even or $n \equiv 3~ (mod~4)$, then $\fZ(\fG) = \{ 0 \}$.  If $n \equiv 1~(mod~4)$, then $\fZ(\fG) = F-span \{\omega \}$
\end{corollary}

\noindent We need a preliminary result

\begin{lemma}   Let $\{e_{1}, ... , e_{n} \}$ be a Q-orthogonal basis of V. Let $c : C\ell(V,Q) \rightarrow C\ell(V,Q)$ be the canonical anti-automorphism.  Let $I = (i_{1}, i_{2}, ... , i_{k}) \in \frak{I}$ be an arbitrary multi-index.  Then

	1) $c(e_{I}) = - e_{I}$ if $|I| \equiv 1$ or $2~mod~4$
	
	2) $c(e_{I}) =  e_{I}$ if $|I| \equiv 0$ or $3~mod~4$
\end{lemma}
	
\begin{proof}  For $|I|$ = 1,2,3 or 4 it is routine to verify the assertion.  Let $|I| \geq 5$ and write $I = I_{1} \cup I_{2}$, where $I_{2} = \{i_{k-3}, i_{k-2}, i_{k-1},i_{k} \}$ and $I_{1} = I - I_{2}$.  Then $e_{I} = e_{I_{1}} \cdot e_{I_{2}}$ and $c(e_{I}) = c(e_{I_{2}})\cdot c(e_{I_{1}}) = e_{I_{2}} \cdot c(e_{I_{1}}) = (-1)^{4|I_{1}|} c(e_{I_{1}}) \cdot e_{I_{2}} = c(e_{I_{1}}) \cdot e_{I_{2}}$.  The assertion now follows by induction since $|I| = |I_{1}| + 4$.
\end{proof}

\noindent  We complete the proof of the Corollary.  If n is even, then $\fZ(\fG) = \fZ(V,Q) \cap \fG = F \cap \fG = \{0\}$ by assertion 1) of Proposition 5.1.  If $n \equiv 3~(mod~4)$, then $c(\omega) = \omega$ by Lemma 6.4, and hence $\omega$ does not lie in $\fG$.  It follows that $\fZ(\fG) = \fZ(V,Q) \cap \fG = \{0 \}$ by assertion 2) of Proposition 5.1.  If $n \equiv 1~(mod~4)$, then $c(\omega) = - \omega$ by Lemma 6.4, and hence $\omega \in \fG$.  We conclude that  $\fZ(\fG) = \fZ(V,Q) \cap \fG = F-span~\{\omega \}$ by assertion 2) of Proposition 5.1.
\newline

\noindent $\mathbf{Killing~ form~ of~ \fG}$
\newline

\noindent Let $\{e_{1}, ... , e_{n} \}$ be a Q-orthogonal basis of V.  Define the Killing form $B : \fG \times \fG \rightarrow F$ by $B(\xi, \eta) = trace~ad~\xi \circ ad~\eta$.  In this section we show that $B(e_{I},e_{J}) = 0$ if $e_{I}, e_{J}$ are distinct elements of $\fG$.  We also compute $B(e_{I},e_{I})$ for all $e_{I} \in \fG$.  Finally, we show that $\fG = \fZ(\fG) \oplus \fH$, where $\fH$ is an ideal of $\fG$ such that B restricted to $\fH$ is nondegenerate.  By Corollary 6.3 it then follows that $\fG = \fH$ if $n \neq~1~(mod~4)$ and $\fH$ has codimension 1 in $\fG$ if $n \equiv 1~(mod~4)$.

\begin{lemma}  $\fG = F-\rm {span} \{e_{I} : e_{I} \in \fG \}$.
\end{lemma}

\begin{proof}  Clearly $\fG \supset F-\rm{span} \{e_{I} : e_{I} \in \fG \}$.  Now let $\xi \in \fG$ and write $\xi = \xi_{0} + \sum_{I \in \frak{I}} \xi_{I} e_{I}$, where $\xi_{0}, \xi_{I} \in F$.  Then  $ - \xi_{0} - \sum_{I \in \frak{I}} \xi_{I} e_{I} = - \xi = c(\xi) =   \xi_{0} + \sum_{I \in \frak{I}} \xi_{I} c(e_{I})$.  By Lemma 6.4  $c(e_{I}) = \pm e_{I}$.  By inspection and the linear independence of $\{e_{I} : I \in \fI \}$  it follows that $\xi_{0} = 0$ and $c(e_{I}) = - e_{I}$ if $\xi_{I} \neq 0$.  We conclude that $\xi \in  F-\rm{span} \{e_{I} : e_{I} \in \fG \}$, which completes the proof.
\end{proof}

\begin{proposition}   Let $\{e_{1}, ... , e_{n} \}$ be a Q-orthogonal basis of V.  Let I,J be distinct multi-indices in $\fI$ such that $e_{I} \in \fG$ and $e_{J} \in \fG$.  Then $B(e_{I}, e_{J}) = 0$.
\end{proposition}

\noindent Again, we need some preliminary results.

\begin{lemma}  Let $\xi \in \fG$.  Then for all $x,y \in \fG$ we have $\overline{Q}(ad~\xi(x),y) + \overline{Q}(x, ad~\xi(y)) = 0$.
\end{lemma}  

\begin{proof}  Note that $ad~\xi = L_{\xi} - R_{\xi}$.  By Proposition 4.1 we have $\overline{Q}(ad~\xi(x) , y) = \overline{Q}(L_{\xi}(x) , y) - \overline{Q}(R_{\xi}(x) , y) = \overline{Q}(x,L_{c(\xi)}(y)) -  \overline{Q}(x,R_{c(\xi)}(y)) = - \overline{Q}(x , L_{\xi}(y)) + \overline{Q}(x , R_{\xi}(y)) = -\overline{Q}(x, (L_{\xi} - R_{\xi})(y)) = - \overline{Q}(x , ad~\xi(y)$.
\end{proof}

\begin{lemma}  Let $I,J \in \frak{I}$ be distinct elements such that $e_{I},e_{J} \in \fG$.  Then \newline $\overline{Q}(ad~e_{I} \circ ad~e_{J}(e_{K} ), e_{K}) = 0$ for all $K \in \frak{I}$ with $e_{K} \in \fG$.  
\end{lemma}

\begin{proof}  $\overline{Q}(ad~e_{I} \circ ad~e_{J}(e_{K} ), e_{K}) = - \overline{Q}(ad~e_{J}(e_{K}), ad~e_{I}(e_{K}))$ by Lemma 6.7.  By Proposition 4.1 and Lemma 6.4 we obtain $\overline{Q}(ad~e_{J}(e_{K}), ad~e_{I}(e_{K})) = \overline{Q}(e_{J}e_{K} - e_{K} e_{J} , e_{I} e_{K} - e_{K} e_{I}) = p([e_{J}e_{K} - e_{K} e_{J}], c([e_{I}e_{K} - e_{K} e_{I}])) = p([e_{J}e_{K} - e_{K} e_{J}], [c(e_{K})c(e_{I}) - c(e_{I})c(e_{K})]) = \alpha_{1} + \alpha_{2} + \alpha_{3} + \alpha_{4}$, where $\alpha_{i} = \lambda_{i} p(e_{J} e_{I})$ for some nonzero $\lambda_{i} \in F$.  However, $p(e_{I} e_{J}) = 0$ since $I \neq J$.
\end{proof}

\noindent  We now complete the proof of the Proposition. Let $K \in \frak{I}$ with $e_{K} \in \fG$.   Since $(ad~ e_{I} \circ ad~e_{J})(e_{K}) \in \fG$ we may write $(ad e_{I} \circ ad~e_{J})(e_{K}) = \sum_{e_{L} \in \fG} A_{IJK}^{L} e_{L}$ for some constants $A_{IJK}^{L} \in F$.   Hence $B(e_{I} , e_{J}) = \sum_{e_{K} \in \fG} A_{IJK}^{K}$.  It suffices to show that $A_{IJK}^{K} = 0$ if $e_{K} \in \fG$.  Recall that $\{1, e_{I}, I \in \frak{I} \}$ is a $\overline{Q}$ - orthogonal basis for $C\ell(V,Q)$ by Proposition 1.3.  Hence by Lemma 6.8 we have $0 = \overline{Q}(ad~e_{I} \circ ad~e_{J}(e_{K}) , e_{K}) = \overline{Q}(e_{K}, e_{K}) A_{IJK}^{K}$.  Note that  $\overline{Q}(e_{K}, e_{K}) \neq 0$ by Proposition 1.3.  We conclude that $A_{IJK}^{K} = 0$ if $I \neq J$ and $e_{K} \in \fG$.  

\begin{proposition} Let $\{e_{1}, ... , e_{n} \}$ be a Q-orthogonal basis of V. Let $I,K \in \frak{I}$ with $e_{I}, e_{K} \in \fG$.  Then $(ad~e_{I})^{2}(e_{K}) = \lambda e_{K}$, where $\lambda = 0$ or $4e_{I}^{2} \in F$. 
\end{proposition}

\begin{corollary}  Let $\{e_{1}, ... , e_{n} \}$ be a Q-orthogonal basis of V. Let $I,J \in \frak{I}$ with $e_{I},e_{J} \in \fG$.  Then 

	1)  $B(e_{I} , e_{J}) = \overline{Q}(e_{I} , e_{J}) = 0$ if $I \neq J$.

	2) $B(e_{I} , e_{I}) = 4 m_{I} e_{I}^{2}$, where $m_{I}$ is the number of $e_{K}$ in $\fG$ such that $(ad~e_{I})^{2}(e_{K})$ is nonzero.
	
	3)  $\overline{Q}(e_{I} , e_{I}) = - e_{I}^{2}$.
\end{corollary}

\begin{proof}  We first prove the corollary.  Assertion 1) follows from Propositions 6.6 and 1.3. We prove 2).  If $e_{K} \in \fG$, then by Proposition 6.9 $(ad~e_{I})^{2}(e_{K}) = \lambda_{K} e_{K}$, where $\lambda_{K} = 0$ or $4e_{I}^{2}$.  Hence $B(e_{I}, e_{I}) = trace~(ad~e_{I})^{2} = \sum_{e_{K} \in \fG} \lambda_{K}$, which completes the proof of 2).  To prove 3) we compute $\overline{Q}(e_{I} , e_{I}) = p(e_{I} c(e_{I})) = p(- e_{I}^{2}) = - e_{I}^{2}$ since $e_{I}^{2} \in F$ by Lemma 4.2.
\newline 

\noindent  We now prove the Proposition.  We show first that $e_{K}$ is an eigenvector of $(ad~e_{I})^{2}$.  Write $(ad~e_{I})^{2}(e_{K}) = \sum_{e_{J} \in \fG} A_{IK}^{J} e_{J}$ for suitable elements $A_{IK}^{J} \in F$. It suffices to show that $A_{IK}^{J} = 0$ if $e_{J} \in \fG, J \neq K$.  By Lemma 6.7 and the argument used in the proof of Lemma 6.8 we have $ \overline{Q}((ad~e_{I})^{2}(e_{K}), e_{J}) = - \overline{Q}(ad~e_{I}(e_{K}), ad~e_{I}(e_{J})) = \lambda~ p(e_{J} e_{K}) = 0$ if $J \neq K$.  On the other hand,  $ \overline{Q}((ad~e_{I})^{2}(e_{K}), e_{J}) = A_{IK}^{J} \overline{Q}(e_{J}, e_{J})$ and $\overline{Q}(e_{J} , e_{J})$ is nonzero.  We conclude that $A_{IK}^{J} = 0$ if $e_{J} \in \fG, J \neq K$.
\newline

\noindent  We have shown that $(ad~e_{I})^{2}(e_{K}) = \lambda_{K} e_{K}$ for some $\lambda_{K} \in F$ if $e_{I}, e_{K} \in \fG$.  Note again that $ad~e_{I} = L_{e_{I}} - R_{e_{I}}$ and hence $(ad~e_{I})^{2} = (L_{e_{I}})^{2} + (R_{e_{I}})^{2} - 2 L_{e_{I}} \circ R_{e_{I}} = (2e_{I}^{2})~Id - 2 L_{e_{I}} \circ R_{e_{I}}$ since $e_{I}^{2} \in F$ by Lemma 4.2.  Hence $e_{K}$ is an eigenvector of $(e_{I})^{2}~Id - \frac{1}{2} (ad~e_{I})^{2} = L_{e_{I}} \circ R_{e_{I}}$.  Now $(L_{e_{I}} \circ R_{e_{I}})^{2} = e_{I}^{4}~Id$, where $0 \neq e_{I}^{2} \in F$,  so the only eigenvalues of $ L_{e_{I}} \circ R_{e_{I}}$ are $e_{I}^{2}$ and $- e_{I}^{2}$.  If $(L_{e_{I}} \circ R_{e_{I}})(e_{K}) = e_{I}^{2}~e_{K}$, then $(ad~e_{I})^{2}(e_{K}) = 2e_{I}^{2} e_{K} - 2e_{I}^{2} e_{K} = 0$.  If $(L_{e_{I}} \circ R_{e_{I}})(e_{K}) = - e_{I}^{2}~e_{K}$, then $(ad~e_{I})^{2}(e_{K}) = 2e_{I}^{2} e_{K} +2e_{I}^{2} e_{K} = 4e_{I}^{2} e_{K}$.
\end{proof}

\begin{proposition}  Let $e_{I} \in \fG$ for some $I \in \frak{I}$.  Then $B(e_{I},e_{I}) = 0 \Leftrightarrow e_{I} = \omega = e_{1} e_{2} ... e_{n}$ and $\omega \in \fZ(V,Q)$. 
\end{proposition}

\begin{proof}  If $e_{I} = \omega \in \fZ(V,Q)$, then $B(e_{I},e_{I}) = trace~(ad~e_{I})^{2} = 0$.  Conversely, suppose that $B(e_{I},e_{I})  = 0$ for some $I \in \frak{I}$ with $e_{I} \in \fG$. Then $(ad~e_{I})^{2} = 0$ on $\fG$ by 2) of Corollary 6.10.  By Propositions 5.1 and  6.2 it suffices to show that $e_{I} \in \fZ(\fG)$.
\newline

\noindent  Let $J \in \frak{I}$ with $e_{J} \in \fG$ be given.  We show that $ad~e_{I}(e_{J}) = 0$, and we may assume that $I \neq J$.  By Lemma 4.2   $e_{J} e_{I} = \lambda e_{I} e_{J}$, where $\lambda = \pm 1$. 
\newline

\noindent  We compute $0 = (ad~e_{I})^{2}(e_{J}) = e_{I}(e_{I} e_{J} - e_{J} e_{I}) - (e_{I} e_{J} - e_{J} e_{I}) e_{I} = e_{I}^{2} e_{J} - 2 e_{I} e_{J} e_{I} + e_{J} e_{I}^{2} = 2 e_{I}^{2} e_{J} - 2 e_{I}(\lambda e_{I} e_{J}) = (2 - 2 \lambda) e_{I}^{2} e_{J}$.  It follows that $\lambda = 1$ since $e_{I}^{2} \neq 0$, and we conclude that $e_{J} e_{I} = e_{I} e_{J}$  or $ad~e_{I}(e_{J}) = 0$.
\end{proof} 

\begin{proposition}  There exists an ideal $\fH$ of $\fG$ such that $\fG = \fZ(\fG) \oplus \fH$ and the restriction of B to $\fH$ is nondegenerate.  If dim V $ \equiv 1~(mod~4)$, then $\fH$ has codimension 1 in $\fG$ but otherwise $\fG = \fH$.
\end{proposition}

\begin{proof}   We prove the first assertion.  Assuming that $\fH$ exists, it follows from Corollary 6.10 and Proposition 6.11 that B is nondegenerate on $\fH$.  If $\fZ(\fG) = \{0 \}$, then set $\fH = \fG$.   If $\fZ(\fG) \neq \{0 \}$, then $\fZ(\fG) = F-span \{ \omega\}$ by Corollary 6.3.  In this case let $\fH = F-span \{e_{I} : e_{I} \in \fG~\rm{and}~e_{I} \neq \omega \}$.  Clearly $\fG = \fZ(\fG) \oplus \fH$, so it remains only to prove that $\fH$ is an ideal of $\fG$.   
\newline

\noindent It suffices to consider the case that $\fZ(\fG) \neq \{ 0\}$. Let $I,J \in \frak{I}$ with $e_{I} \in \fH, e_{J} \in \fG$ and $e_{J} \neq \omega$ be given.  It suffices to prove that $ad~\omega(e_{I}) \in \fH$ and $ad~e_{J}(e_{I}) \in \fH$.  The first assertion is obviously true since $\omega \in \fZ(\fG)$.  To prove the second assertion write $ad~e_{J}(e_{I}) = \xi + \lambda \omega$ for some $\xi \in \fH$ and some $\lambda \in F$.  This is possible since $\fG = \fH \oplus \fZ(\fG)$ and  $ad~e_{J}(e_{I}) \in \fG$.  We need to show that $\lambda = 0$.  By Corollary 6.10  and the definition of $\fH$ we obtain $\overline{Q}(\omega, \xi) = 0$.  We now compute $0 = - \overline{Q}(e_{I} , ad~e_{J}(\omega)) = \overline{Q}(ad~e_{J}(e_{I}) , \omega) = \overline{Q}(\xi , \omega) + \lambda \overline{Q}(\omega, \omega) =  \lambda \overline{Q}(\omega, \omega)$.  It follows that $\lambda = 0$ since $\overline{Q}(\omega , \omega) \neq 0$ by assertion 3) of Corollary 6.10.
\newline

\noindent  The argument above proves the first assertion of the Proposition.  The second assertion follows immediately from the first and from Corollary 6.3.
\end{proof}

\section{Quaternion linear algebra}

\noindent  Let $\bh$ denote the quaternions, and let $\bh^{n}$ denote the space of n-tuples of quaternions.  The space $\bh^{n}$ is an $\bh$-module, where the elements of $\bh$ act by multiplication on the left.  Let $M(n,\bh)$ denote the n x n matrices with entries in $\bh$.  Linear algebra in $M(n, \bh)$ is slightly more complicated than in $M(n,\bc)$ or $M(n,\br)$ since $\bh$ is noncommutative.  Some concepts, such as the determinant, don$^\prime$t exist in $M(n,\bh)$.  A good reference for this section is [R].
\newline

\noindent  If $A = (A_{ij}) \in M(n,\bh)$ and $x = (x_{1}, ... , x_{n}) \in \bh^{n}$, then we define $A (x) = (y_{1}, ... , y_{n})$, where $y_{i} = \sum_{k=1}^{n} x_{k} A_{ik}$.  For elements $A,B \in M(n,\bh)$ we define $A \cdot B \in M(n,\bh)$ by $(A \cdot B)_{ij} = \sum_{k=1}^{n} B_{kj} A_{ik}$.  For real or complex matrices this is just the usual definition of matrix multiplication.   It follows that $M(n,\bh)$ is an algebra.  Next, define the metric adjoint operation $^{*}$ in $M(n,\bh)$ by $A^{*}_{ij} = \overline{A_{ji}}$, where $x \rightarrow \overline{x}$  denotes conjugation in $\bh$. 
\newline

\noindent  Define $\br$-linear maps I,J,K on $\bh^{n}$ by $I(x) = i x, J(x) = j x$ and $K(x) = k x$.  It follows from Lemma 7.3 below that the elements of $M(n,\bh)$ commute with I,J and K.
\newline

\noindent  Let $GL(n,\bh)$ denote the set of invertible elements of $M(n,\bh)$.  It is evident that $GL(n,\bh)$ is a group.  It is also known that $GL(n,\bh)$ is a dense open subset of $M(n,\bh)$ (see for example part b) of Proposition 5.10 of [ R]).
\newline

\noindent It is straightforward to verify the first three of  the following statements.

\begin{lemma}  Let $A,B \in M(n,\bh)$ and let $x \in \bh^{n}$.  Then $(A \cdot B) x = A(B x)$
\end{lemma}

\begin{lemma}  Let $A,B \in M(n,\bh)$.  Then $(A \cdot B)^{*} = B^{*} \cdot A^{*}$.
\end{lemma}

\begin{lemma}  Let $x \in \bh, y \in \bh^{n}$ and $A \in M(n,\bh)$.  Then $A(xy) = x A(y)$.
\end{lemma}

\begin{lemma}  Let $\fB = \{u_{1}, ... , u_{n} \}$ be an basis of $\bh^{n}$ as an $\bh$-module.  Given $A \in M(n,\bh)$ let $\fB(A) \in M(n,\bh)$ be the unique matrix such that $A(u_{i}) = \sum_{r=1}^{n} \fB(A)_{ri}~ u_{r}$.  Then the map $\fB : M(n,\bh) \rightarrow M(n,\bh)$ is an algebra isomorphism.
\end{lemma}

\begin{proof}   It is straightforward to show that $\fB$ is an $\bh$-linear isomorphism of $M(n,\bh)$ and we omit the details.  We prove that $\fB$ preserves multiplication. Let $A,B \in M(n,\bh)$ be given.  Then $ (A \cdot B)(u_{i}) = \sum_{s=1}^{n} \fB(A \cdot B)_{si}~ u_{s}$.  On the other hand by Lemmas 7.1 and 7.3 we have $(A \cdot B)(u_{i}) = A(B(u_{i})) = A(\sum_{r=1}^{n} \fB(B)_{ri}~u_{r}) = \sum_{r=1}^{n} \fB(B)_{ri}~A(u_{r}) = \sum_{r=1}^{n} \fB(B)_{ri} (\sum_{s=1}^{n} \fB(A)_{sr}~u_{s}) = \newline \sum_{s=1}^{n} (\sum_{r=1}^{n} \fB(B)_{ri} \fB(A)_{sr})~u_{s} = \sum_{s=1}^{n} (\fB(A) \cdot \fB(B))_{si}~u_{s}$.  Since $\{ u_{1}, ... , u_{n}\}$ is an $\bh$-basis for $\bh^{n}$ it follows that   $\fB(A \cdot B)_{si} =  (\fB(A) \cdot \fB(B))_{si}$ for all $ 1\leq i,s \leq n$
\end{proof}

\section{Classical Clifford algebras and matrix algebras}

\noindent Let $F = \br$ and let $Q_{1}, Q_{2} : \br^{n}  \times \br^{n} \rightarrow \br$ be symmetric, positive definite bilinear forms. By Corollary 2.2 $Q_{1}$ and $Q_{2}$ are equivalent, and $C\ell(\br^{n}, Q_{1})$ is algebra isomorphic to $C\ell(\br^{n}, Q_{2})$.  We denote this isomorphism class of Clifford algebras by $C\ell(n)$. The traditional explicit model for $C\ell(n)$ is to define Q on $\br^{n}$ by $Q(e_{i},e_{j}) = \delta_{ij}$, where $\{e_{1}, e_{2}, ... , e_{n} \}$ is the usual basis of $\br^{n}$.  In particular $e_{i}^{2} = - 1$ for $1 \leq i \leq n$.  Let $G_{n} = U(\br^{n}, Q)$ in the notation of Proposition 6.1.  If $\fG_{n} = \{\xi \in C\ell(\br^{n},Q) : c(\xi) = - \xi \}$, then $\fG_{n}$ is the Lie algebra of $G_{n}$ by the discussion in section 6.  Let $\fH_{n}$ be the ideal of $\fG_{n}$ defined in Proposition 6.12.
\newline

\begin{proposition}  The Killing form B is negative definite on $\fH_{n}$.
\end{proposition}

\begin{proof}  Let $I = (i_{1}, i_{2}, ... , i_{k})$ be a arbitrary multi-index in $\fI$, where $1 \leq k = |I| \leq n$.  By 2) of Corollary 6.10 and Proposition 6.11 it suffices to show that $e_{I}^{2} = -1$ for all $e_{I} \in \fH_{n}$.  By 1) of Lemma 4.2  $e_{I}^{2} = (-1)^{\frac{k(k-1)}{2}} e_{i_{1}}^{2} e_{i_{12}}^{2} ... e_{i_{k}}^{2} = (-1)^{\frac{k(k+1)}{2}}$.  The fact that $e_{I} \in \fH_{n}$ means that $k = |I| \equiv 1$ or $2~(mod~4)$ by Lemma 6.4.  In either case it follows immediately that  $e_{I}^{2} = (-1)^{\frac{k(k+1)}{2}} = -1$.
\end{proof}

\noindent  From 1) and 3) of Corollary 6.10 we now obtain

\begin{corollary}  The symmetric, bilinear form $\overline{Q}$ is positive definite on $\fH_{n}$.
\end{corollary}

\noindent  Next we describe $C\ell(n)$ as a matrix algebra over $K = \br,\bc$ or $\bh$.  Given K and an integer n let M(n,K) denote the K-algebra of  n x n matrices with elements in K.  We list the isomorphism classes of the classical Clifford algebras.  For details see Table 1 and the discussion in Chapter 1, section 4 of  [LM].  The algebra isomorphism $C\ell(n+8) \approx C\ell(n) \otimes M(16,\br)$ plays a key role. 

\begin{proposition}  The Clifford algebras $C\ell(n)$ are algebra isomorphic to matrix algebras A as given below :
\end{proposition}

\noindent 1) $C\ell(8k) \hspace {.27in}  \hspace{1in}  A = M(2^{4k}, \br)$

\noindent 2) $C\ell(8k+1) \hspace{1in} A = M(2^{4k}, \bc)$
 
 \noindent 3) $C\ell(8k+2)  \hspace{1in} A = M(2^{4k}, \bh)$

\noindent 4) $C\ell(8k+3) \hspace{1in}  A = M(2^{4k}, \bh) \oplus  M(2^{4k}, \bh)$

\noindent 5) $C\ell(8k+4) \hspace{1in} A = M(2^{4k+1}, \bh)$ 

\noindent 6) $C\ell(8k+5) \hspace{1in} A = M(2^{4k+2}, \bc)$

\noindent 7) $C\ell(8k+6) \hspace{1in} A = M(2^{4k+3}, \br)$
 
\noindent 8) $C\ell(8k+7) \hspace{1in} A = M(2^{4k+3}, \br) \oplus M(2^{4k+3}, \br)$

\section {A special isomorphism between C$\ell(n)$ and a matrix algebra}

\begin{proposition}  Let A denote the matrix algebra, or sum of matrix algebras, in the list 1) through 8) of Proposition 8.3. In each of these cases there exist algebra isomorphisms $\rho : C \ell(8k + \alpha) \rightarrow A$, $0 \leq \alpha \leq 7$  such that $\rho(c(x)) = \rho(x)^{*}$ for all $x \in C \ell(8k + \alpha)$
\end{proposition} 

\noindent $\mathbf{Remarks}$  

	1) In $ M(n,K) \oplus M(n,K)$  we define the operation $^{*}$ in the natural way, namely, $(X,Y)^{*} = (X^{*}, Y^{*})$ for all $(X,Y) \in M(n,K) \oplus M(n,K)$.
	
	2)  The isomorphism $\rho : C\ell(8k + \alpha) \rightarrow A$ with the properties stated in Proposition 9.1 is not unique.  Let $\rho$ be one such isomorphism, and let g be an element of A such that $1 = g \cdot g^{*} = g^{*} \cdot g$.  If $\rho^{\prime} : C\ell(8k + \alpha,0) \rightarrow A$ is given by $\rho^{\prime}(x) = g \cdot \rho(x) \cdot g^{*}$, then $\rho^{\prime}$ is another such isomorphism.

\begin{proof}   
\noindent  The proof is essentially the same in all cases.  We give the proof only in the most difficult cases 3), 4) and 5), where the division algebra K in question is the quaternions $\bh$.  In each of these three cases we begin with a fixed isomorphism $\sigma : C\ell(8k + \alpha) \rightarrow A$, $\alpha = 2,3,4$.   Throughout the proof we define $p = 2^{4k}$ if $n = 8k+2, p = 2^{4k}$ if $n = 8k+3$ and $p = 2^{4k+1}$ if $n = 8k+4$.  Note that $\sigma(\br^{n})$ acts on $\bh^{p}$ since $\sigma(\br^{n}) \subset \sigma(C\ell(n)) = M(p,\bh)$ or $M(p,\bh) \oplus M(p,\bh)$.  
\newline

\noindent We now break the proof into three steps.  In all of them we regard $\bh^{p}$ as a real vector space of dimension 4p.

\begin{lemma}  There exists a positive definite inner product $\langle , \rangle$ on $\bh^{p}$ such that the elements of $\sigma(\br^{n})$ are skew symmetric relative to $\langle , \rangle$ and the elements $\{I,J,K \}$ in $GL(\br,\bh^{p})$ are both skew symmetric and orthogonal  relative to $\langle , \rangle$.
\end{lemma}

\begin{lemma}  Let $\langle , \rangle$ be a positive definite inner product on $\bh^{p}$ as in Lemma 9.2.  Then there exists an orthonormal $\br$- basis $\fB^{\prime} = \{u_{1}, ... , u_{4p} \}$ of $\bh^{p}$ such that for $1 \leq r \leq p$ we have $u_{p+r} = I(u_{r}), u_{2p+r} = J(u_{r})$ and $u_{3p+r} = K(u_{r})$. 
\end{lemma}

\noindent  \noindent $\mathbf{Remark}$ Let $\fB^{\prime} = \{u_{1}, ... , u_{4p} \}$ be an orthonormal $\br$-basis of $\bh^{p}$ as in Lemma 9.3.  Then the first p elements $\{u_{1}, ... , u_{p} \}$ become in a natural way a basis of $\bh^{p}$ as a free $\bh$-module.  Let u be an element of $\bh^{p}$.  By Lemma 9.3 there exist unique  real numbers $\alpha_{r}, \beta_{r}, \gamma_{r}, \delta_{r}, 1 \leq r \leq p$ such that $u = \sum_{r=1}^{p} \alpha_{r} u_{r} + \sum_{r=1}^{p} \beta_{r} I(u_{r}) + \sum_{r=1}^{p} \gamma_{r} J(u_{r}) + \sum_{r=1}^{p} \delta_{r} K(u_{r})$.  Now write $u = \sum_{r=1}^{p} h_{r} u_{r}$, where $h_{r} = \alpha_{r} + i \beta_{r} + j \gamma_{r} + k \delta_{r}$ for $1 \leq r \leq p$.  Conversely, given elements $h_{1}, ... , h_{r} \in \bh$ we reverse the argument above to define $u = \sum_{r=1}^{p} h_{r} u_{r} \in \bh^{p}$.
\newline

\begin{lemma}  Choose a positive definite inner product $\langle , \rangle$ induced by $\sigma$ and an orthonormal basis $\fB^{\prime} = \{u_{1}, ... , u_{4p} \}$ as in Lemma 9.2 and 9.3.  Let $\fB : A \rightarrow A$ be the isomorphism induced by the $\bh$-basis $\fB = \{u_{1}, ... , u_{p} \}$ of $\bh^{p}$ as in Lemma 7.4.   Let $\rho = \fB \circ \sigma : C\ell(n) \rightarrow A$.  Then $\rho$ is an isomorphism such that $\rho(x)^{*} = - \rho(x)$ for all $x \in \br^{n}$.
\end{lemma}

\noindent $\mathbf{Remark }$ In case 4) it would be more precise to say that $\fB$ is the diagonal isomorphism $\fB \times \fB$ induced by $\fB = \{u_{1}, ... , u_{p} \}$.
\newline

\noindent  For the moment we assume that the three lemmas above have been proved, and we use them to prove the Proposition.   Let $A^{\prime}  = \{x \in C\ell(n) : \rho(x)^{*} = \rho(c(x)) \}$.  Since $\rho : C\ell(n) \rightarrow A $ is an algebra isomorphism and c and the transpose operation $^{*}$ are algebra anti-automorphisms it follows  that $A^{\prime}$ is a subalgebra of $C\ell(n)$.  Note that $A^{\prime} \supset \br^{n}$ by Lemma 9.4.  Hence by 2) of Corollary 1.4 we conclude that $A^{\prime} = C\ell(n)$.
\newline

\noindent $\mathit{Proof~ of~ Lemma~ 9.2}$  We consider first the algebra isomorphisms $\sigma : C\ell(8k+2) \rightarrow A = M(2^{4k},\bh)$ and $C\ell(8k+4) \rightarrow A = M(2^{4k+1},\bh)$.  For notational simplicity we let n = 8k+2 and p = $2^{4k}$ in the first case, and n = 8k+4 and p = $2^{4k+1}$ in the second case.
\newline

\noindent Recall that Pin(n) is the set of finite products $x = x_{1} ... x_{m}$ in $C\ell(n)$, where each $ x_{i}$ is  a  unit~ vector  in  $\br^{n}$ and m is an  arbitrary  positive  integer.  The group Pin(n) is compact (cf. Chapter 1, section 1 of [LM]) and it is easy to check that $x \cdot c(x) = c(x) \cdot x = 1$ for all $x \in Pin(n)$ since $x_{i}^{2} = - 1$ for all i. Note that the unit vectors in $\br^{n}$ lie in Pin(n).
\newline

\noindent Let $H = Pin(n)_{0}$, the identity component of Pin(n), and let $x_{1}, ... , x_{N}$ be elements of Pin(n) such that $Pin(n) = \bigcup_{i=1}^{N} x_{i} H$.  Let dH (Haar measure) denote the measure on H induced from the unique bi-invariant volume form $\Omega$ such that $\int_{H}\Omega = 1$.  One of the basic properties of Haar measure dH is that $\int_{H} (f \circ R_{h})~dH =  \int_{H} (f \circ L_{h})~dH = \int_{H} f~ dH$ for all smooth functions $f : H \rightarrow \br$ and all elements h of H.
\newline 

\noindent Now let $\langle , \rangle_{0}$ be an arbitrary positive definite inner product on $\bh^{p}$.  Fix $u,v \in \bh^{p}$ and define $\langle u , v \rangle_{1} = \int_{h \in H} \langle \sigma(h)(u), \sigma(h)(v) \rangle_{0}~dH$.  By the left invariance property of the Haar measure it follows that $\langle , \rangle_{1}$ is a positive definite inner product on $\bh^{p}$ that is preserved by $\sigma(H)$.  Finally, define $\langle u,v \rangle_{2} = \sum_{i=1}^{N} \langle \sigma(x_{i})(u), \sigma(x_{i})(v) \rangle_{1}$.  The inner product $\langle , \rangle_{2}$ is preserved by the elements  $\{ \sigma(x_{1}), ... , \sigma(x_{N}) \}$.  If $h \in H$, then $h_{i} = x_{i} h x_{i}^{-1} \in H$ for $1 \leq i \leq N$.  Hence $\langle \sigma(h) u , \sigma(h) v \rangle_{2} = \sum_{i=1}^{N} \langle \sigma(x_{i}h) u , \sigma(x_{i}h) v \rangle_{1} =    \sum_{i=1}^{N} \langle \sigma(h_{i} x_{i}) u , \sigma(h_{i} x_{i}) v \rangle_{1} =  \sum_{i=1}^{N} \langle \sigma(x_{i}) u , \sigma(x_{i}) v \rangle_{1} = \langle u , v \rangle_{2}$.  This shows that $\langle , \rangle_{2}$ is preserved by $\sigma(H)$.  We conclude that $\langle , \rangle_{2}$ is preserved by $\sigma(Pin(n))$ and in particular by the elements $\sigma(x)$, where x is a unit vector of $\br^{n}$.
\newline

\noindent To prove that $\sigma(x)$ is skew symmetric for all $x \in \br^{n}$ it suffices to consider the case that x is a unit vector.  If x is a unit vector in $\br^{n}$, then $x^{2} = -1$ and hence $\sigma(x)^{2} = - Id$.  The map $\sigma(x)$ is orthogonal relative to $\langle , \rangle _{2}$, and hence it is also skew symmetric  since $\langle \sigma(x) u, v \rangle_{2} = \langle \sigma(x)^{2} u, \sigma(x) v \rangle_{2} = - \langle u , \sigma(x) v \rangle_{2}$ for all u,v $\in \bh^{p}$.
\newline

\noindent Let $C = \{ \pm{Id}, \pm{I}, \pm{J}, \pm{K}\}$.  Note that C is a subgroup of eight elements in $GL(\br, \bh^{p})$ that commutes with the elements of $M(p,\bh) = \sigma(C\ell(n))$ by Lemma 7.3.  For $u,v \in \bh^{p}$ define $\langle u , v \rangle = \sum_{q \in C} \langle q(u) , q(v) \rangle_{2}$.  The inner product $\langle , \rangle$ is positive definite and preserved by both $\sigma(Pin(n))$ and Q since these two groups commute.  Since $I^{2} = J^{2} = K^{2} = - Id$ it follows as above that I,J and K are skew symmetric as well as orthogonal relative to $\langle , \rangle$.
\newline

\noindent  Next we consider the isomorphism $\sigma : C\ell(8k+3) \rightarrow A = M(2^{4k}, \bh) \oplus M(2^{4k}, \bh)$.  Again, for notational simplicity we set $n = 8k+3$ and $p = 2^{4k}$ in the discussion below. 
\newline 

\noindent Let $p_{1} : A \rightarrow M(p, \bh)$ and $p_{2} : A \rightarrow M(p, \bh)$ denote the projections onto the first and second $M(p, \bh)$ factors of A respectively.  From the isomorphism $\sigma : C\ell(n) \rightarrow A$ we obtain the algebra homomorphisms $\sigma_{1} = p_{1} \circ \sigma : C\ell(n) \rightarrow M(p, \bh)$ and $\sigma_{2} = p_{2} \circ \sigma : C\ell(n) \rightarrow M(p, \bh)$.  If $G_{1} = \sigma_{1}(Pin(n))$ and $G_{2} = \sigma_{2}(Pin(n))$, then $G_{1}$ and $ G_{2}$ are compact subgroups of $GL(p,\bh)$ that commute since the first and second factors of $M(p, \bh)$ in A  commute. 
\newline

\noindent   Let G denote the compact Lie group $G_{1} \times G_{2}$.  Note that $\sigma(S^{n-1}) \subset \sigma(Pin(n)) \subset \sigma_{1}(Pin(n)) \times  \sigma_{2}(Pin(n)) = G$. The group G acts on $\bh^{p}$ by $(g_{1}, g_{2})(x) = g_{2}(g_{1}(x)) = g_{1}(g_{2}(x))$.  Now let $\langle, \rangle_{0}$ be an arbitrary positive definite inner product on $\bh^{p}$.  Average it over G, following the argument above, to obtain a positive definite G-invariant inner product $\langle , \rangle_{1}$ on $\bh^{p}$ such that the elements of $\sigma(\br^{n})$ are skew symmetric relative to  $\langle , \rangle_{1}$.  Average  $\langle , \rangle_{1}$ over the finite group C to obtain an inner product $\langle , \rangle$.  The group C commutes with the elements of $M(n,\bh) \times M(n,\bh)$ by Lemma 7.3 and in particular with the elements of G and $\sigma(S^{n-1}) \subset G$.  Hence $\langle , \rangle$ is G invariant, and by the argument above $\langle , \rangle$ satisfies the conditions of Lemma 9.2.
\newline

\noindent $\mathit{Proof~of~Lemma~9.3}$  Let $\langle , \rangle$ be an inner product on $\bh^{p}$ as in Lemma 9.2.  Let $u_{1}$ be a unit vector, and let $U_{1} = \br$-span $\{u_{1}, I(u_{1}), J(u_{1}), K(u_{1}) \}$.  Note that  $\{u_{1}, I(u_{1}), J(u_{1}), K(u_{1}) \}$ is an orthonormal basis of $U_{1}$ since the transformations $I,J,K$ are skew symmetric and orthogonal.  Now consider the orthogonal complement $U_{1}^{\perp}$ of $U_{1}$ in $\bh^{p}$.  The elements of $C = \{\pm Id, \pm I, \pm J, \pm K \}$ leave $U_{1}^{\perp}$ invariant by Lemma 9.2.  Let $u_{2}$ be a unit vector in $U_{1}^{\perp}$ and define $U_{2} = \br$-span $\{u_{2}, I(u_{2}), J(u_{2}), K(u_{2}) \} \subset U_{1}^{\perp}$.   Proceed in this fashion to obtain a real orthonormal basis of $\bh^{p}$ of the form $\{u_{1}, I(u_{1}), J(u_{1}), K(u_{1}) \} \cup \{u_{2}, I(u_{2}), J(u_{2}), K(u_{2}) \} \cup ... \cup \{u_{p}, I(u_{p}), J(u_{p}), K(u_{p}) \}$.  We now obtain the desired basis $\fB^{\prime}$ by defining $u_{p+r} = I(u_{r}), u_{2p+r} = J(u_{r})$ and $u_{3p+r} = K(u_{r})$ for $1 \leq r \leq p$.
\newline

\noindent $\mathit{Proof~of~Lemma~ 9.4}$  Let $\langle , \rangle$ be a positive definite inner product  on $\bh^{p}$  induced by $\sigma$  as in Lemma 9.2. Let $\fB = \{u_{1}, ... , u_{p} \}$ be the $\bh$ basis of $\bh^{p}$ defined in Lemma 9.3 and the following discussion.  Let $x \in \br^{n}$ be given.  Then for $ 1 \leq r \leq p$ we have $\sigma(x)(u_{r}) = \sum_{s=1}^{p} \fB(\sigma(x))_{sr}~u_{s} = \sum_{s=1}^{p} \rho(x)_{sr}~u_{s}$.  For $1 \leq r,s \leq p$ we have $\rho(x)_{sr} = \langle \sigma(x)(u_{r}) , u_{s} \rangle$, which lies in $\br$ since $\langle , \rangle$ has real values on $\bh^{p}$.  By Lemma 9.2 we know that $\sigma(x) : \bh^{p} \rightarrow \bh^{p}$ is skew symmetric relative to $\langle , \rangle$.  For $1 \leq r,s \leq p$ it follows that $\rho(x)^{*}_{rs} = \overline{\rho(x)_{sr}} = \rho(x)_{sr} = \langle \sigma(x)(u_{r}) , u_{s} \rangle = - \langle u_{r}, \sigma(x)(u_{s}) \rangle = - \rho(x)_{rs}$.  This completes the proof.
\end{proof}

\begin{theorem}   Let $k \geq 0$ be an integer, and let $\approx$ denote group isomorphism.   For a positive integer n let $G_{n} = \{ g \in C\ell(n) : g \cdot c(g) = c(g) \cdot g = 1\}$.  Then 

	$1)  G_{8k} \approx O(2^{4k}, \br)$
	
	$2)  G_{8k+1} \approx U(2^{4k})$
	
	$3)  G_{8k+2} \approx Sp(2^{4k})$
	
	$4)  G_{8k+3} \approx Sp(2^{4k}) \times Sp(2^{4k})$
	
	$5)  G_{8k+4} \approx Sp(2^{4k+1})$
	
	$6)  G_{8k+5} \approx U(2^{4k+2})$
	
	$7)  G_{8k+6} \approx  O(2^{4k+3}, \br)$
	
	$8)  G_{8k+7} \approx O(2^{4k+3}, \br) \times O(2^{4k+3}, \br)$
	
\noindent In particular, $G_{n}$ is compact for all positive integers n.
\end{theorem}

\noindent $\mathbf{Remark}$  The groups $G_{8k+1}$ and $G_{8k+5}$ have 1-dimensional center since the groups $U(2^{4k})$ and $U(2^{4k+2})$ have this property.  (See also Corollary 6.3).  In all other cases $G_{n}$ is a semisimple group.

\begin{proof}  Let $\rho : C\ell(n) \rightarrow A$ be an algebra isomorphism as in Proposition 9.1 such that $\rho(c(g)) = \rho(g)^{*}$ for all $g \in C\ell(n)$.    Then $g \in G_{n} \Leftrightarrow 1 = g \cdot c(g) = c(g) \cdot g \Leftrightarrow 1 = \rho(g) \cdot \rho(c(g)) = \rho(c(g)) \cdot \rho(g) \Leftrightarrow 1 = \rho(g) \rho(g)^{*} = \rho(g)^{*} \rho(g)$.  For $K = \br, \bc , \bh$ let $U(n,K) = \{g \in M(n,K) : g \cdot g^{*} = g^{*} \cdot g = 1 \}$.  Then $U(n,K) = O(n, \br)$ if $K = \br$, U(n,K) = U(n) if $K = \bc$ and $U(n,K) = Sp(n)$ if $K = \bh$.
\newline

\noindent  The discussion above shows that $\rho(C\ell(n)) = U(n,K)$ if $n \neq 3~(mod~4)$ and $\rho(C\ell(n)) = U(n,K) \times U(n,K)$ if $n \equiv 3~(mod~4)$.  The eight assertions above now follow directly from the corresponding eight assertions in Proposition 8.3.
\end{proof}

\noindent $\mathbf{References}$
\newline

\noindent [FH] W. Fulton and J. Harris, "Representation Theory, A First Course", Springer, New York, 1991.
\newline

\noindent [H]  F.R. Harvey, "Spinors and Calibrations", Perspectives in Mathematics, vol.9, Academic Press, New York, 1990.
\newline

\noindent [L]  S. Lang, "Algebra" (revised Third Edition), Springer, New York, 2002.
\newline

\noindent [LM] H. B. Lawson and M-L. Michelsohn, "Spin Geometry", Princeton University Press, Princeton, 1989.
\newline

\noindent [R]  L. Rodman, "Topics in Quaternion Linear Algebra", Princeton Series in Applied Mathematics, Princeton University Press, Princeton, 2014.

\end{document}